\documentclass[a4paper,12pt]{amsart}
\usepackage[
	hmarginratio={1:1},
	vmarginratio={1:1},
	textwidth=15.5cm,
	textheight=21cm,
	heightrounded
]{geometry}

\usepackage[utf8]{inputenc}
\usepackage{amsfonts,amstext,amsmath,amsthm,amscd,amssymb}
\usepackage[sc]{mathpazo}
\usepackage[dvipsnames]{xcolor}
\usepackage[pagebackref=false]{hyperref}
\usepackage{url}

\definecolor{dark-red}{rgb}{0.4,0.15,0.15}
\definecolor{dark-blue}{rgb}{0.15,0.15,0.4}
\definecolor{dark-green}{rgb}{0.15,0.4,0.15}
\hypersetup{
	colorlinks,
	linkcolor=dark-red,
	citecolor=dark-blue,
	urlcolor=dark-green
}

\numberwithin{equation}{section}

\theoremstyle{plain}
\newtheorem{theorem}{Theorem}[section]
\newtheorem{proposition}[theorem]{Proposition}
\newtheorem{lemma}[theorem]{Lemma}
\newtheorem{corollary}[theorem]{Corollary}

\theoremstyle{definition}
\newtheorem{definition}[theorem]{Definition}
\newtheorem{remark}[theorem]{Remark}

\newcommand{\DD}{\mathbb{D}}
\newcommand{\BB}{\mathbb{B}}
\newcommand{\s}{\mathbb{S}}
\newcommand{\ZZ}{\mathbb{Z}}
\newcommand{\As}{\operatorname{As}}
\newcommand{\Cos}{\operatorname{Cos}}
\newcommand{\Alex}{\mathcal{A}}

\title[Same $(\pi_1,\pi_2,Q)$, different first Postnikov invariants]
{Fundamental Quandles Do Not Determine the First Postnikov Invariant of 2-Knots}

\author{Micha{\l} Jab{\l}onowski}

\address{Institute of Mathematics, Faculty of Mathematics, Physics and
Informatics, University of Gda\'nsk, 80-308 Gda\'nsk, Poland}

\email{\href{mailto:michal.jablonowski@gmail.com}
{michal.jablonowski@gmail.com}}


\date{\today}

\begin{document}

\begin{abstract}
For every admissible pair of Brieskorn parameters in the
Plotnick--Suciu construction, we obtain a pair of oriented $2$-knots
whose knot groups are isomorphic, whose second homotopy modules are
semilinearly isomorphic under a suitable group isomorphism, and whose
fundamental quandles are isomorphic, while no compatible group and
module isomorphisms carry one first Postnikov invariant to the other.
Consequently, their exteriors are not homotopy equivalent. Thus, the
knot group, the second homotopy module up to semilinear equivalence,
and the fundamental quandle do not determine the homotopy type of an
oriented $2$-knot exterior.
Using an alternative construction from Suciu's thesis based on
punctured lens spaces, the same peripheral argument yields, for every
$N\geq2$, a family of $N$ oriented $2$-knots with these properties.

We additionally show that the Tanaka--Taniguchi examples with
isomorphic knot groups and distinct fundamental quandles have pairwise
inequivalent second homotopy modules: no isomorphism between two of the
knot groups makes the corresponding second homotopy modules
semilinearly isomorphic.
\end{abstract}

\maketitle

\section{Introduction}
\label{sec:introduction}

The fundamental quandle of a codimension-two knot is a meridional
invariant which retains information that is not visible from the
abstract knot group alone. In the classical setting, the knot quandle
was introduced independently by Joyce \cite{Joy82} and Matveev
\cite{Mat82}; the same noose construction applies to knotted surfaces
and, more generally, to higher-dimensional codimension-two knots
\cite{CKS04,Win09}. For an oriented $2$-knot
$K\subset \s^4$, the fundamental quandle is determined by the knot
group together with its oriented meridional data.

This raises a natural question about the position of the fundamental
quandle relative to the homotopy-theoretic information carried by the
exterior. For a connected CW complex $X$, the algebraic $2$-type is
$
\mathcal T(X)=\bigl(\pi_1(X),\pi_2(X),k_X\bigr),
$
where $\pi_2(X)$ is a $\ZZ[\pi_1(X)]$-module and
$
k_X\in H^3\bigl(\pi_1(X);\pi_2(X)\bigr)
$
is the first Postnikov invariant. The module $\pi_2$ and the class
$k_X$ arise from related chain-level information, but $k_X$ is not
determined by the isomorphism class of the pair $(\pi_1,\pi_2)$
\cite{Lom81,Plot83,PS85}. The main question of this paper is whether
adjoining the fundamental quandle to that pair can recover the missing
Postnikov information.

The examples come from the construction of Plotnick and Suciu
\cite{PS85}. For suitable fibered $2$-knots, their construction gives
two exteriors $\widehat X,\widehat X'$ with common fundamental group
$
\Pi
=
\operatorname{BS}(1,2)*_{\langle x\rangle}H,
\;
\operatorname{BS}(1,2)
=
\langle t,x\mid txt^{-1}=x^2\rangle,
$
and isomorphic $\ZZ[\Pi]$-modules $\pi_2$, while their first Postnikov
classes lie in distinct orbits under all compatible group and module
isomorphisms. The essential additional observation made here is that
the corresponding oriented $2$-knots also have isomorphic fundamental
quandles.

The point is geometric and peripheral. Both Plotnick--Suciu
constructions contain the same boundary piece
$
W=\s^1_t\times\BB^3.
$
After choosing a common orientation of the distinguished meridional
loop
$
\mu_t=\s^1_t\times\{p'\}
$
and orienting the two final $2$-knots so that $\mu_t$ is positive in
both cases, the based identifications with $\Pi$ send both positive
meridians to the same literal element $t\in\Pi$. Consequently, the two
oriented peripheral triples are represented by
$
\bigl(\Pi,\langle t\rangle,t\bigr),
$
and the peripheral coset description of the fundamental quandle gives
$
Q(\widehat K)
\cong
\Cos\bigl(\Pi,\langle t\rangle,t\bigr)
\cong
Q(\widehat K').
$

Combining this observation with the Plotnick--Suciu calculation yields
the main result, Theorem~\ref{thm:main-PS}: there are oriented
$2$-knots whose compatible pairs $(\pi_1,\pi_2)$ are isomorphic and
whose fundamental quandles are isomorphic, but whose first Postnikov
invariants are inequivalent under every compatible group and module
isomorphism.  Thus, knowing both the isomorphism class of the compatible
pair $(\pi_1,\pi_2)$ and the isomorphism class of the fundamental
quandle does not determine the homotopy type of the exterior, and hence
does not determine the first Postnikov invariant.

Suciu's thesis \cite[Chapter~III]{Suc84} gives an independent
construction of arbitrarily large families of $2$-knots with common
fundamental group and isomorphic second homotopy modules but pairwise
inequivalent first Postnikov invariants. In that construction the
exteriors are obtained by surgery in
$\s^1_t\times\BB^3\,\#\,Y_{p,q}$ on the same word
$txt^{-1}x^{-2}$, and the distinguished meridian coming from the
$\s^1_t\times\BB^3$ piece is represented by the same element $t$
throughout \cite[Chapter~III, \S3]{Suc84}. Applying the peripheral
argument above therefore shows that the fundamental quandles also
agree; see Corollary~\ref{cor:arbitrarily-many}.

It is important that the orientation change used by Plotnick and
Suciu occurs in an internal closed fiber summand. Their two inputs use
$\Sigma_1$ and $-\Sigma_1$, respectively, in order to change the
Postnikov class. This should not be confused with reversing the
orientation of the final embedded $2$-sphere. The present paper uses
the former operation only; orientation reversal and non-reversibility
of the final $2$-knots are not part of the main argument.

The question of how much information the fundamental quandle contains
beyond the abstract knot group has recently been investigated
explicitly for $2$-knots. Tanaka and Taniguchi \cite{TT26} constructed
infinitely many triples
$
F_p=\tau^p(T_{q,r}),\;
F_q=\tau^q(T_{r,p}),\;
F_r=\tau^r(T_{p,q}),
$
for suitable pairwise coprime integers $p,q,r>1$, such that
$
G(F_p)\cong G(F_q)\cong G(F_r),
$
whereas their fundamental quandles are mutually non-isomorphic. Their
distinction is detected by quandle type:
$
\operatorname{type}Q(F_p)=p,\;
\operatorname{type}Q(F_q)=q,\;
\operatorname{type}Q(F_r)=r.
$
Thus, the abstract knot group does not determine the fundamental
quandle.

We show that these examples do not already realize the phenomenon of
the present paper. Although their second homotopy groups are mutually
isomorphic as abstract abelian groups, the corresponding
$\ZZ[\pi_1]$-module structures are pairwise inequivalent under every
isomorphism of the knot groups. Hence, the compatible pairs
$(\pi_1,\pi_2)$ already distinguish the three Tanaka--Taniguchi
examples. This is the content of
Theorem~\ref{thm:TT-pi2-modules-distinct}.

A related comparison is provided by Suciu's family of fibered ribbon
$2$-knots. Suciu constructed infinitely many ribbon $2$-knots $R_k$
whose knot groups are all isomorphic to the trefoil knot group while
their second homotopy modules are pairwise non-isomorphic
\cite{Suc85}. Their fundamental quandles are also mutually
non-isomorphic \cite{Jab25,Yas25,ST26}. Again, therefore, the
$\pi_2$-module already separates the examples. The Plotnick--Suciu pair
considered here occupies a different position: both $\pi_1$ and
$\pi_2$ agree, and the fundamental quandle also agrees, but the first
Postnikov invariant does not.

A closely related meridional phenomenon appears in recent work of
Bais et al.~\cite{BDPetal25}, who construct non-isotopic
codimension-two knots with orientation-preservingly diffeomorphic
traces and distinguish them by finite-group representation counts in
which the conjugacy class of the meridian is prescribed. In the
$2$-knot case such counts also obstruct isomorphism of the fundamental
quandles: a quandle isomorphism induces, through the associated group,
an isomorphism of knot groups carrying a positive meridian to a
conjugate of a positive meridian, and therefore preserves every such
meridian-prescribed representation count. This gives another
illustration that the abstract knot group alone does not retain the
relevant meridional marking.

The Alexander module does not provide additional information beyond
the fundamental quandle in this comparison. Indeed, we show in
Proposition~\ref{prop:quandle-determines-Alexander} that the Alexander
module of an oriented $2$-knot is determined by its fundamental quandle. 
For both the Brieskorn Plotnick--Suciu pair and Suciu's lens-space
family considered below, the common module can be computed explicitly.
For the former,
$
\Alex(\widehat K)
\cong
\Alex(\widehat K')
\cong
\ZZ[t^{\pm1}]/(t-2),
$
and Corollary~\ref{cor:suciu-alexander} gives the same module for
every member of Suciu's lens-space family.
Thus, the examples agree not only in their fundamental groups, second
homotopy modules, and fundamental quandles, but also in their
Alexander modules and Alexander polynomials. Nevertheless, their
first Postnikov invariants remain inequivalent.

The organization of the paper is as follows.
Section~\ref{sec:basic-invariants} fixes the noose convention, the
fundamental quandle, the associated group, the peripheral triple, and
the algebraic $2$-type, and recalls the peripheral coset description
of the fundamental quandle.
Section~\ref{sec:tanaka-taniguchi-pi2} analyzes the
Tanaka--Taniguchi examples and proves that their compatible second
homotopy modules are pairwise non-isomorphic.
Section~\ref{sec:plotnick-suciu} recalls the relevant part of the
Plotnick--Suciu construction, specializes it to Brieskorn homology
spheres, identifies the common positive meridian, and proves the main
result that the two knots have the same compatible $(\pi_1,\pi_2)$ data
and fundamental quandle but inequivalent first Postnikov invariants.
It then applies the same peripheral argument to Suciu's lens-space
construction to obtain arbitrarily large families with the same
properties, and concludes by computing the Alexander modules of both
families explicitly.


\section{Basic invariants and peripheral data}
\label{sec:basic-invariants}

Let $K\subset \s^4$ be an oriented smooth $2$-knot (in the standard $\s^4$) and let
$
X_K=\s^4\setminus\operatorname{int}N(K)
$
be its compact exterior. Put
$
G_K=\pi_1(X_K),
\;
R_K=\ZZ[G_K].
$
If a basepoint $x_0\in X_K$ is fixed, the pair
$
(X_K,x_0)
$
is called the \emph{pointed exterior} of $K$.
The inclusion $X_K\hookrightarrow \s^4\setminus K$ is a homotopy
equivalence, so either exterior model may be used to compute homotopy
groups.

We freely pass between the standard $\mathbb R^4$ and $\s^4$ descriptions of
surface-knots when convenient; the corresponding ambient-isotopy
classes are naturally identified
\cite[\S1.5]{Kam17}.

\subsection{Nooses and fundamental quandles}
\label{subsec:nooses-fundamental-quandles}

We fix the following conventions. If paths $\gamma$ and $\delta$
are composable, then
$
\gamma\cdot\delta
$
means that $\gamma$ is traversed first and $\delta$ second. The
boundary of an oriented disk $D$ is always given its boundary
orientation.

\begin{definition}[Quandle]
	\label{def:quandle}
	A \emph{quandle} is a nonempty set $Q$ together with a binary operation
	$
	*\colon Q\times Q\to Q
	$
	satisfying the following conditions for all $x,y,z\in Q$:
	\begin{enumerate}
		\item
		$
		x*x=x;
		$
		\item the right translation
		$
		S_y\colon Q\to Q,
		\;
		S_y(x)=x*y,
		$
		is a bijection;
		\item
		$
		(x*y)*z=(x*z)*(y*z).
		$
	\end{enumerate}
\end{definition}

\begin{definition}[Type of a quandle]
	Let $Q$ be a quandle, and for $y\in Q$ let
	$
	S_y\colon Q\to Q,
	\;
	S_y(x)=x*y.
	$
	The \emph{type} of $Q$ is
	$
	\operatorname{type}(Q)
	=
	\min\bigl\{
	n\geq 1
	\mid
	S_y^n=\operatorname{id}_Q
	\text{ for every }y\in Q
	\bigr\},
	$
	if such an integer exists, and
	$
	\operatorname{type}(Q)=\infty
	$
	otherwise.
\end{definition}

\begin{definition}[Associated group of a quandle]
\label{def:associated-group}
Let $Q=(Q,*)$ be a quandle. Its \emph{associated group} is
$
\As(Q)
=
\left\langle e_x\ (x\in Q)
\;\middle|\;
 e_{x*y}=e_y^{-1}e_xe_y
\text{ for all }x,y\in Q
\right\rangle.
$

A quandle homomorphism $f\colon Q\to Q'$ induces a group homomorphism
$$
\As(f)\colon\As(Q)\to\As(Q'),
\;
 e_x\mapsto e_{f(x)},
$$
so $\As(-)$ is functorial.
\end{definition}

Let
$
K\subset \s^4
$
be a smooth $2$-knot, let $N(K)\cong K\times \DD^2$ be a closed tubular
neighbourhood, and let
$
X_K=\s^4\setminus\operatorname{int}N(K)
$
be its compact exterior. Fix a basepoint
$
x_0\in X_K.
$

\begin{definition}[Meridional disk]
	\label{def:meridional-disk}
	A \emph{meridional disk} of $K$ is a smoothly embedded disk
	$
	D\subset N(K)
	$
	such that:
	\begin{enumerate}
		\item $D$ meets $K$ transversely in exactly one point;
		\item
		$
		D\cap K=\operatorname{int}D\cap K;
		$
		\item
		$
		\partial D\subset\partial N(K)\subset X_K.
		$
	\end{enumerate}
	The orientation of $D$ induces an orientation on its boundary
	$\partial D$.
\end{definition}

\begin{definition}[Noose]
	\label{def:noose}
	A \emph{noose} of $K$, based at $x_0$, is a pair
	$
	(D,\alpha),
	$
	where $D$ is an oriented meridional disk and
	$
	\alpha\colon[0,1]\to X_K
	$
	is a path satisfying
	$
	\alpha(0)\in\partial D,
	\;
	\alpha(1)=x_0.
	$
	Thus, $\alpha$ is oriented from the meridional disk towards the
	basepoint.
	
	Two nooses $(D_0,\alpha_0)$ and $(D_1,\alpha_1)$ are equivalent if
	there exists a homotopy through nooses
	$
	(D_s,\alpha_s),\; s\in[0,1],
	$
	such that the terminal endpoint of every $\alpha_s$ is the basepoint
	$x_0$, and the orientations of the meridional disks are preserved
	throughout the homotopy. 
	The homotopy class of a
	noose is denoted by
	$
	[(D,\alpha)].
	$
\end{definition}

The noose $(D,\alpha)$ determines the based meridian
$
m(D,\alpha)
=
\alpha^{-1}\cdot\partial D\cdot\alpha
\in\pi_1(X_K,x_0),
$
where $\partial D$ denotes the boundary loop beginning and ending at
$\alpha(0)$.

An orientation of $K$, together with the fixed
orientation of $\s^4$, selects one of the two possible orientations of a
meridional disk.

\begin{definition}[Positive meridional disk]
	\label{def:positive-meridional-disk}
	Suppose that $K$ is oriented. The normal bundle of $K$ is oriented by
	the convention
	$
	\operatorname{or}(TK)\wedge
	\operatorname{or}(\nu K)
	=
	\operatorname{or}(T\s^4).
	$
	An oriented meridional disk $D$ is called \emph{positive} if its
	orientation agrees with the induced orientation of the normal bundle
	$\nu K$. Otherwise, it is called \emph{negative}.
\end{definition}

\begin{definition}[Fundamental quandle of an oriented $2$-knot]
	\label{def:oriented-fundamental-quandle}
	Let $K\subset \s^4$ be an oriented $2$-knot. Its
	\emph{fundamental quandle}, or \emph{knot quandle}, denoted by $Q(K)$,
	is the set of homotopy classes of nooses $(D,\alpha)$ whose meridional
	disk $D$ is positively oriented. The quandle operation is defined by
	$$
	\begin{aligned}
\bigl[(D_1,\alpha_1)\bigr]
*
\bigl[(D_2,\alpha_2)\bigr]
=
\bigl[
\bigl(
D_1,\,
\alpha_1\cdot\alpha_2^{-1}\cdot\partial D_2\cdot\alpha_2
\bigr)
\bigr].
	\end{aligned}
	$$
	Equivalently, the path of the first noose is followed by the based
	meridian determined by the second noose.
\end{definition}

This is the standard noose description of the fundamental quandle;
see \cite[\S4.1.2]{CKS04} and \cite[\S8.9]{Kam17}. Our path convention agrees
with the convention used there.

\begin{remark}
	\label{rem:quandle-basepoint-independence}
	Changing the basepoint and choosing a path between the old and new
	basepoints induces a quandle isomorphism. Thus, the isomorphism type of
	$Q(K)$ is independent of the chosen basepoint.
\end{remark}

\begin{definition}[Second homotopy module]
	\label{def:second-homotopy-module}
	The group $\pi_2(X_K)$ is regarded as a left $R_K$-module via
	the action of deck transformations on the universal cover
	$\widetilde X_K$. Since $\widetilde X_K$ is simply connected, the
	Hurewicz theorem gives a natural isomorphism
	$
	\pi_2(X_K)\cong H_2(\widetilde X_K;\ZZ)
	$
	of left $R_K$-modules.
\end{definition}

\begin{definition}[Algebraic $2$-type]
	\label{def:algebraic-2-type}
	The \emph{algebraic $2$-type} of a connected CW complex $X$ is the
	triple
	$
	\mathcal T(X)=\bigl(\pi_1(X),\pi_2(X),k_X\bigr),
	$
	where $\pi_2(X)$ carries its natural
	$\ZZ[\pi_1(X)]$-module structure and
	$
	k_X\in H^3\bigl(\pi_1(X);\pi_2(X)\bigr)
	$
	is the first Postnikov invariant \cite[\S~0]{Lom81}.
	
	Terminology varies in the literature: following Lomonaco \cite{Lom81}
	we call this triple the \emph{algebraic $2$-type}; authors who index
	the terminology by the degree of the first $k$-invariant sometimes
	refer to the same data as an algebraic $3$-type.
\end{definition}

The Postnikov $2$-type and its boundary-sensitive refinements also
play a central role in recent homotopy-classification results for
$4$-manifolds with boundary; see, for example,
\cite{CK25}.

If $G$ is a group and $M$ a left $\ZZ[G]$-module, a compatible
automorphism of the pair $(G,M)$ means a pair
$(\alpha,\beta)$, where
$
\alpha\in\operatorname{Aut}(G)
$
and $\beta\colon M\to M$ is an additive automorphism satisfying
$
\beta(g\cdot u)=\alpha(g)\cdot\beta(u)
$
for all $g\in G$ and $u\in M$.

\begin{remark}
	\label{rem:chain-level-k}
	For a finite $3$-complex, the cellular chain complex of the universal
	cover contains
	$
	C_3(\widetilde X)
	\xrightarrow{\partial_3}
	C_2(\widetilde X)
	\xrightarrow{\partial_2}
	C_1(\widetilde X).
	$
	The second homotopy module is
	$
	\pi_2(X)
	\cong
	\ker(\partial_2)/\operatorname{im}(\partial_3).
	$
	The first Postnikov invariant is the next obstruction to extending the
	resulting partial free resolution. Therefore, the module
	$\pi_2(X)$ and the class $k_X$ depend on related chain-level data, but
	$k_X$ is not determined merely by the isomorphism class of the quotient
	module $\pi_2(X)$ \cite[\S~1]{PS85};
	\cite[Theorem~7.1]{Lom81}.
\end{remark}

\begin{definition}[Peripheral triple]
\label{def:peripheral-triple}
For an oriented $2$-knot $K$, let
$$
P_K=\operatorname{im}\bigl(\pi_1(\partial X_K)\to G_K\bigr)
$$
and choose a positive meridian $m_K\in P_K$. The ordered triple
$
\mathcal P(K)=(G_K,P_K,m_K)
$
is called the \emph{oriented peripheral triple}. Changing the basepoint
paths conjugates $P_K$ and $m_K$ simultaneously.
\end{definition}

\begin{definition}[Peripheral coset quandle]
\label{def:coset-quandle}
Let $G$ be a group, let $P\leq G$, and let $m\in Z(P)$. The
\emph{peripheral coset quandle}
$
\Cos(G,P,m)
$
has underlying set $P\backslash G$ and operation
$
(Pg)*_m(Ph)=P\bigl(gh^{-1}mh\bigr).
$
In the $2$-knot case, $P=\langle m\rangle$, so the centrality hypothesis is
automatic.
\end{definition}

\begin{proposition}[Peripheral coset description]
\label{prop:quandle-peripheral}
Let $K\subset \s^4$ be an oriented $2$-knot with peripheral triple
$\mathcal P(K)=(G_K,P_K,m_K)$. The map
$$
\eta_K\colon\As(Q(K))\to G_K,
\;
 e_{[(D,\alpha)]}
\mapsto
\alpha^{-1}\cdot\partial D\cdot\alpha,
$$
is an isomorphism, and there is a natural quandle isomorphism
$
Q(K)\cong\Cos(G_K,P_K,m_K).
$
For a $2$-knot,
$
\partial X_K\cong \s^2\times \s^1,
\;
P_K=\langle m_K\rangle\cong\ZZ.
$
Thus, the fundamental quandle is determined by the knot group together with
the conjugacy class and orientation of its meridian
\cite{Joy82,FR92, Win09}.
\end{proposition}

\begin{proof}
	For a noose $q=[(D,\alpha)]$, write
	$
	m(q)=\alpha^{-1}\cdot\partial D\cdot\alpha\in G_K.
	$
	The noose operation satisfies
	$
	m(q_1*q_2)=m(q_2)^{-1}m(q_1)m(q_2),
	$
	so the assignment
	$
	e_q\mapsto m(q)
	$
	respects the defining relations of $\As(Q(K))$.
	
	For codimension-two knots, the standard Wirtinger--noose construction
	identifies the associated group of the fundamental quandle with the knot
	group; see \cite{Joy82,FR92,Win09} and also
	\cite[Proposition~8.9.7]{Kam17}. Hence, the resulting homomorphism
	$
	\eta_K\colon\As(Q(K))\to G_K
	$
	is an isomorphism. Equivalently, passing from a quandle presentation to
	the associated-group presentation replaces each relation $u*v=w$ by the
	conjugation relation $v^{-1}uv=w$.
	
For the coset description, use the standard path model in which the
path $\alpha$ of a noose, oriented from the meridional disk towards
the basepoint, represents a coset $P_Kg\in P_K\backslash G_K$. Changing the initial
point of the path along $\partial X_K$ changes $g$ by left
multiplication by an element of $P_K$, so the noose class determines
an element of $P_K\backslash G_K$.

With this convention, the noose corresponding to $P_Kg$ has based
meridian
$
g^{-1}m_Kg.
$
Hence, if $q_1$ and $q_2$ correspond respectively to $P_Kg$ and
$P_Kh$, then the path defining $q_1*q_2$ is obtained from the first
path by appending the based meridian
$h^{-1}m_Kh$ determined by the second noose. Therefore
$
(P_Kg)*(P_Kh)
=
P_K\bigl(gh^{-1}m_Kh\bigr).
$
This is precisely
$
\Cos(G_K,P_K,m_K).
$
	
	Finally,
	$
	\pi_1(\partial X_K)\cong\ZZ,
	$
	its image is generated by the meridian, and the meridian has infinite
	order because its image generates
	$
	H_1(X_K;\ZZ)\cong\ZZ.
	$
	Hence
	$
	P_K=\langle m_K\rangle\cong\ZZ.
	$
\end{proof}

\begin{remark}[Alexander-module convention]
	Throughout, the positive generator $t$ of the deck group acts by
	the deck transformation associated with a positive meridian;
	equivalently, on the abelianisation of $\ker\varepsilon_K$ it acts
	by
	$
	[h]\longmapsto[m_Khm_K^{-1}].
	$
	With the opposite deck-generator convention one replaces $t$ by
	$t^{-1}$. Thus, the module written later as $\Lambda/(t-2)$ would
	be written as $\Lambda/(2t-1)$. These two presentations are exchanged
	by the involution $t\mapsto t^{-1}$, although $t-2$ and $2t-1$ are
	not associates in the fixed Laurent polynomial ring
	$\Lambda=\ZZ[t^{\pm1}]$.
\end{remark}

\begin{proposition}[The Alexander module is determined by the fundamental quandle]
	\label{prop:quandle-determines-Alexander}
	Let $K\subset\s^4$ be an oriented $2$-knot, let
	$
	G_K=\pi_1(X_K),
	\;
	\varepsilon_K\colon G_K\to\ZZ
	$
	be the abelianisation normalized by $\varepsilon_K(m_K)=1$, let
$X_K^\infty\to X_K$ be the associated infinite cyclic cover, on whose
homology the deck group acts, and let
$
\Alex(K)=H_1(X_K^\infty;\ZZ)
$
be its Alexander module over
$\Lambda=\ZZ[t^{\pm1}]$.

Since $X_K$ has finite CW type, $\Alex(K)$ is a finitely presented
$\Lambda$-module.
		Write $\Delta_K(t)$ for a generator, defined up to multiplication by
	a unit $\pm t^j$ of $\Lambda$, of the smallest principal ideal
	containing the zeroth Fitting ideal of $\Alex(K)$.
	Put
	$
	\Gamma_K=\As(Q(K)),
	$
	and let
	$
	\varepsilon_Q\colon\Gamma_K\to\ZZ
	$
	be the homomorphism determined by
	$\varepsilon_Q(e_q)=1$ for every $q\in Q(K)$.
	Then
	$$
	\Alex(K)
	\cong
	\frac{\ker\varepsilon_Q}
	{[\ker\varepsilon_Q,\ker\varepsilon_Q]}
	$$
	as $\Lambda$-modules, where for any $q\in Q(K)$ the action of $t$
is given by
$
t\cdot[h]=[e_qhe_q^{-1}].
$
	
	Consequently, an isomorphism
	$
	Q(K_1)\cong Q(K_2)
	$
	induces an isomorphism
	$
	\Alex(K_1)\cong \Alex(K_2)
	$
	of $\Lambda$-modules. In particular,
	$
	\Delta_{K_1}(t)\doteq\Delta_{K_2}(t).
	$
\end{proposition}

\begin{proof}
	The relations $e_{x*y}=e_y^{-1}e_xe_y$ of Definition~\ref{def:associated-group}
	have equal total exponent on both sides, so $\varepsilon_Q$ is well defined.
	By Proposition~\ref{prop:quandle-peripheral} the canonical map
	$\eta_K\colon\As(Q(K))\to G_K$ is an isomorphism sending every canonical
	generator $e_q$ to a positive based meridian. All positive based meridians
	are conjugate in $G_K$, hence have the same image under $\varepsilon_K$, and
	that image is $\varepsilon_K(m_K)=1$. Therefore
	$\varepsilon_K\circ\eta_K=\varepsilon_Q$, and $\eta_K$ identifies
	$\ker\varepsilon_Q$ with $\ker\varepsilon_K$ and the corresponding derived
	subgroups with one another.
	
The infinite cyclic cover $X_K^\infty\to X_K$ is the connected regular
covering associated with $\ker\varepsilon_K$. Fixing a basepoint
$\tilde x_0$ over $x_0$, the projection identifies
$\pi_1(X_K^\infty,\tilde x_0)$ with $\ker\varepsilon_K$, so the Hurewicz
theorem in degree one gives
$$
\Alex(K)
=
H_1(X_K^\infty;\ZZ)
\cong
\frac{\ker\varepsilon_K}
{[\ker\varepsilon_K,\ker\varepsilon_K]}.
$$
For $g\in G_K$ let $\tau_g$ be the deck transformation with
$\tau_g(\tilde x_0)=\tilde g(1)$, where $\tilde g$ is the lift starting at
$\tilde x_0$ of a loop representing $g$; this depends only on the coset
$g\ker\varepsilon_K$, and $g\ker\varepsilon_K\mapsto\tau_g$ is an
isomorphism $G_K/\ker\varepsilon_K\cong\operatorname{Deck}$. Transporting
along $\tilde g$ identifies $\pi_1(X_K^\infty,\tau_g(\tilde x_0))$ with
$\pi_1(X_K^\infty,\tilde x_0)=\ker\varepsilon_K$, and under this
identification $(\tau_g)_*$ is $h\mapsto ghg^{-1}$. Passing to first
homology, $\tau_g$ acts by the automorphism induced by conjugation by $g$.
Since $t$ acts as the deck transformation determined by the positive
meridian $m_K$, and $\varepsilon_K(m_K)=1$, the $\Lambda$-action on
$\Alex(K)$ is conjugation by any element of $\varepsilon_K$-level one.
	Transporting this identification through $\eta_K$ gives
$
\Alex(K)
\cong
\frac{\ker\varepsilon_Q}
{[\ker\varepsilon_Q,\ker\varepsilon_Q]}.
$
	
	If $q,q'\in Q(K)$, then
	$\varepsilon_Q(e_q)=\varepsilon_Q(e_{q'})=1$, so
	$e_{q'}^{-1}e_q\in\ker\varepsilon_Q$. Conjugation by an element of
	$\ker\varepsilon_Q$ acts trivially on its abelianisation. Hence
	conjugation by $e_q$ and by $e_{q'}$ induce the same automorphism of
	the displayed quotient, giving the well-defined $\Lambda$-action.
	
	Finally, a quandle isomorphism
	$f\colon Q(K_1)\to Q(K_2)$ induces an isomorphism
	$\As(f)$ preserving the canonical augmentations. It therefore induces
	an isomorphism of the corresponding kernel abelianisations, and this
	isomorphism intertwines the action of $t$. Thus, it is
	$\Lambda$-linear.
\end{proof}

\begin{remark}[Attribution]
	\label{rem:alexander-quandle-known}
	For classical knots, the relation between the knot quandle and
	Alexander information goes back to Joyce \cite[\S17]{Joy82} and
	Matveev \cite[\S11]{Mat82}; see also Traldi \cite{Tra22}.
	In categorical language, Szymik showed that the Beck abelianisation
	of the classical knot quandle is the extended Alexander module
	\cite[Proposition~3.2]{Szy19}. The latter fits into the Crowell exact
	sequence with the usual Alexander module \cite{Cro71}.
	Proposition~\ref{prop:quandle-determines-Alexander} gives the direct
	group-theoretic formulation needed here for oriented $2$-knots,
	using the canonical identification
	$\As(Q(K))\cong G_K$.
\end{remark}


\section{The Tanaka--Taniguchi examples and their second homotopy modules}
\label{sec:tanaka-taniguchi-pi2}

Let $p,q,r>1$ be pairwise coprime integers, and define
$
F_p=\tau^p(T_{q,r}),
\;
F_q=\tau^q(T_{r,p}),
\;
F_r=\tau^r(T_{p,q}),
$
where $\tau^n(K)$ denotes the $n$-twist spin of a classical knot
$K$.

For
$
s\in\{p,q,r\},
$
write $F_s$ for the corresponding member of this triple and
$
X_s=X_{F_s}
$
for its exterior.

By Zeeman's twist-spinning theorem \cite{Zee65}, the exterior of
$\tau^s(K)$ fibers over $\s^1$, with fiber the punctured $s$-fold cyclic
branched cover of $\s^3$ along $K$ and with monodromy induced by the
canonical deck transformation.

The cyclic branched covers which occur as the closed fibers of these
three twist spins are all homeomorphic to the same Brieskorn manifold
$
M=\Sigma(p,q,r).
$

For each $s\in\{p,q,r\}$, let
$
d_s\colon M\to M
$
denote the canonical deck transformation associated with the
corresponding cyclic branched covering. Choose a $d_s$-invariant ball
$
\BB_s^3\subset M
$
centered at a fixed point of $d_s$, and put
$
V_s=M\setminus\operatorname{int}\BB_s^3.
$
Choose a basepoint
$
v_s\in\partial V_s
$
fixed by $d_s$.

All the pointed manifolds $(V_s,v_s)$ are homeomorphic. Since
$3$-manifolds have unique smooth structures up to diffeomorphism, these
homeomorphisms may be chosen smooth; after adjusting the marked boundary
point if necessary, we may therefore fix a pointed punctured manifold
$(V,v)$ and pointed diffeomorphisms
$
(V_s,v_s)\xrightarrow{\cong}(V,v).
$

Transporting the restrictions of the deck transformations under these
identifications gives based self-diffeomorphisms
$
f_s\colon (V,v)\to(V,v).
$
With these choices, the exterior $X_s$ is identified with the mapping
torus
$
X_s\cong V\times_{f_s}\s^1.
$

Since deleting an open $3$-ball does not change the fundamental group,
the inclusion $V_s\hookrightarrow M$, and hence also our chosen
identification with $V$, gives
$
A:=\pi_1(V,v)\cong\pi_1(M).
$
Put
$
\varphi_s=(f_s)_*\in\operatorname{Aut}(A).
$
\begin{theorem}[Tanaka--Taniguchi]
	\label{thm:TT-basic}
	The knot groups
	$
	\pi_1(X_p),\;
	\pi_1(X_q),\;
	\pi_1(X_r)
	$
	are mutually isomorphic. On the other hand,
	$
	\operatorname{type}Q(F_p)=p,
	\;
	\operatorname{type}Q(F_q)=q,
	\;
	\operatorname{type}Q(F_r)=r.
	$
	In particular, the three fundamental quandles are mutually
	non-isomorphic.
	
	Moreover, for every $s\in\{p,q,r\}$,
	$
	\operatorname{ord}(\varphi_s)=s.
	$
\end{theorem}

\begin{proof}
	The first assertion is \cite[Theorem~3.1]{TT26}.
	The statement about the quandle types follows from
	\cite[Theorem~3.3]{TT26}, while the equality
	$
	\operatorname{ord}(\varphi_s)=s
	$
	is \cite[Proposition~3.2]{TT26}.
\end{proof}

\subsection{The common underlying second homotopy group}

\begin{proposition}
	\label{prop:TT-common-abstract-pi2}
	The groups
	$
	\pi_2(X_p),
	\pi_2(X_q),
	\pi_2(X_r)
	$
	are mutually isomorphic as abstract abelian groups.
\end{proposition}

\begin{proof}
	For every $s\in\{p,q,r\}$, the mapping-torus structure gives a
	fibration
	$
	V\to X_s\to \s^1.
	$
	The corresponding homotopy exact sequence contains
	$
	\pi_3(\s^1)
	\to
	\pi_2(V)
	\to
	\pi_2(X_s)
	\to
	\pi_2(\s^1).
	$
	Since
	$
	\pi_2(\s^1)=\pi_3(\s^1)=0,
	$
	the inclusion of the fiber induces an isomorphism
	$
	\pi_2(V)\xrightarrow{\cong}\pi_2(X_s).
	$
	The same punctured Brieskorn manifold $V$ occurs for all three
	values $s=p,q,r$, and hence the three abstract second homotopy
	groups are mutually isomorphic.
\end{proof}

\subsection{The common fiber module}

Since $\widetilde V$ is simply connected, the Hurewicz theorem gives
$
\pi_2(V)\cong H_2(\widetilde V;\ZZ).
$
We now describe this group together with the natural $A$-action.

For this computation fix one
$
s\in\{p,q,r\}.
$
Using the pointed diffeomorphism chosen above, identify
$
(V,v)\cong(V_s,v_s),
\;
V_s=M\setminus\operatorname{int}\BB_s^3.
$
For the remainder of this computation we suppress the index $s$ on the
deleted ball and write
$
\BB^3=\BB_s^3.
$
Thus, under the above identification, $V$ is identified with
$
M\setminus\operatorname{int}\BB^3.
$
Choose one lift $B_1$ of the deleted ball
$
\BB^3\subset M.
$
The other lifts are indexed by $A$:
$
B_a=aB_1,
\; a\in A.
$
Let
$
\theta_a\in H_2(\widetilde V;\ZZ)
$
denote the homology class of the oriented boundary sphere
$\partial B_a$.

\begin{lemma}
	\label{lem:TT-fiber-module}
	Let
	$
	P=H_2(\widetilde V;\ZZ).
	$
	Then
	$
	P\cong
	\begin{cases}
		\ZZ[A]/\ZZ N_A, & A \text{ finite},\\
		\ZZ[A], & A \text{ infinite},
	\end{cases}
	\;
	N_A=\displaystyle\sum_{a\in A}a.
	$
	Under these identifications, $\theta_a$ corresponds to the class of
	$a\in\ZZ[A]$.
	
	In particular, in both cases,
	\begin{equation}
		\label{eq:theta-distinct}
		\theta_a=\theta_b
		\;\Longrightarrow\;
		a=b.
	\end{equation}
	
	The $A$-action is given by
	$
	a\cdot\theta_h=\theta_{ah},
	\; a,h\in A.
	$
\end{lemma}

\begin{proof}
	The inverse image of the deleted ball $\BB^3\subset M$ in the universal
	cover $\widetilde M$ is the disjoint union
	$
	\coprod_{a\in A}B_a,
	$
	where $B_a=aB_1$. Hence
	$
	\widetilde V
	=
	\widetilde M\setminus
	\bigcup_{a\in A}\operatorname{int}B_a.
	$
	By excision,
	$
	H_3(\widetilde M,\widetilde V;\ZZ)
	\cong
	\bigoplus_{a\in A}
	H_3(B_a,\partial B_a;\ZZ)
	\cong
	\ZZ[A].
	$
	The corresponding part of the homology exact sequence is
	$
	H_3(\widetilde M;\ZZ)
	\to
	\ZZ[A]
	\to
	H_2(\widetilde V;\ZZ)
	\to
	H_2(\widetilde M;\ZZ).
	$
	
	If $A$ is finite, then $\widetilde M$ is a closed, simply connected,
	oriented $3$-manifold. Hence
	$
	H_3(\widetilde M;\ZZ)\cong\ZZ,
	\;
	H_2(\widetilde M;\ZZ)
	\cong
	H^1(\widetilde M;\ZZ)
	=
	0.
	$
	Under the excision identification above, the fundamental class
	$[\widetilde M]$ maps to the sum of the relative fundamental classes
	of the lifted balls:
	$
	[\widetilde M]
	\longmapsto
	N_A=\sum_{a\in A}a.
	$
	Therefore
	$
	H_2(\widetilde V;\ZZ)
	\cong
	\ZZ[A]/\ZZ N_A.
	$
	
	If $A$ is infinite, then the Brieskorn homology sphere
	$M=\Sigma(p,q,r)$ is aspherical; for pairwise coprime
	$p,q,r>1$, the only spherical case is the Poincar\'e sphere
	$\Sigma(2,3,5)$; see, for example, \cite{Mil75}. Hence
	$
	H_2(\widetilde M;\ZZ)
	=
	H_3(\widetilde M;\ZZ)
	=
	0,
	$
	and consequently
	$
	H_2(\widetilde V;\ZZ)\cong\ZZ[A].
	$
	
	It remains to verify \eqref{eq:theta-distinct}. If $A$ is infinite,
	the identification
	$
	H_2(\widetilde V;\ZZ)\cong\ZZ[A]
	$
	makes this immediate. If $A$ is finite, equality
	$\theta_a=\theta_b$ would imply
	$
	a-b\in\ZZ N_A,
	$
	which is impossible for $a\neq b$ by comparison of coefficients.
	Thus
	$
	\theta_a=\theta_b\Longrightarrow a=b
	$
	in either case.
	
	Finally, a deck transformation $a\in A$ sends
	$
	B_h\longmapsto B_{ah},
	$
	and therefore
	$
	a\cdot\theta_h=\theta_{ah}.
	$
\end{proof}

\subsection{The action of the full knot group}

Let $\mu_s$ be the positive meridian of $F_s$. The mapping-torus
description gives
$
G_s:=\pi_1(X_s)
\cong
A\rtimes_{\varphi_s}\langle\mu_s\rangle,
$
where
$
\mu_s a\mu_s^{-1}=\varphi_s(a),
\; a\in A.
$

Although Gordon's description \cite{Gor72} gives an abstract
isomorphism of these knot groups with $A\times\ZZ$, we retain the
mapping-torus semidirect-product presentation because it records the
distinguished meridian $\mu_s$ and its specific action $\varphi_s$ on
the fiber group. It is this marked action, rather than merely the
abstract isomorphism type of $G_s$, that is used below.

Thus, we write elements of $G_s$ as pairs
$
(a,n),
\;
a\in A,\; n\in\ZZ,
$
with $\mu_s=(1,1)$.

\begin{lemma}
	\label{lem:TT-full-action}
	Under the identification
	$
	\pi_2(X_s)
	\cong
	H_2(\widetilde V;\ZZ)=P,
	$
	the action of $G_s$ is
	$
	(a,n)\cdot\theta_h
	=
	\theta_{a\varphi_s^n(h)}
	$
	for all
	$
	a,h\in A,\; n\in\ZZ.
	$
\end{lemma}
\begin{proof}
The inclusion of the fiber
$
V\hookrightarrow X_s
$
induces the isomorphism
$
\pi_2(V)\xrightarrow{\cong}\pi_2(X_s)
$
from Proposition~\ref{prop:TT-common-abstract-pi2}. We use this
isomorphism to transport the natural $G_s=\pi_1(X_s)$-action on
$\pi_2(X_s)$ to
$
H_2(\widetilde V;\ZZ).
$
The subgroup $A=\pi_1(V)$ acts by deck transformations, while the
meridian $\mu_s$, corresponding to the positive generator of the
mapping-torus direction, acts through a lift of the monodromy.
The action of $A$ was computed in Lemma~\ref{lem:TT-fiber-module}:
$
a\cdot\theta_h=\theta_{ah}.
$

The monodromy of the $s$-twist spin is induced by the canonical
generator of the cyclic transformation group of the branched covering
$M\to \s^3$. The corresponding lift
$
\widetilde f_s\colon\widetilde V\to\widetilde V
$
which fixes the chosen lift of the basepoint preserves the corresponding
boundary component $\partial B_1$. Moreover, with the convention
$
\mu_s a\mu_s^{-1}=\varphi_s(a),
$
it satisfies
$
\widetilde f_s(hx)
=
\varphi_s(h)\widetilde f_s(x)
$
for every $h\in A$ and $x\in\widetilde V$.
Consequently,
$
\widetilde f_s(B_h)=B_{\varphi_s(h)}.
$
Since the canonical deck transformation is orientation preserving, the
induced map preserves the orientations of the boundary spheres. Hence
$
\mu_s\cdot\theta_h
=
\theta_{\varphi_s(h)}.
$
Iterating gives
$
\mu_s^n\cdot\theta_h
=
\theta_{\varphi_s^n(h)}.
$
Combining this with the $A$-action yields
$
(a,n)\cdot\theta_h
=
\theta_{a\varphi_s^n(h)}.
$
\end{proof}

\subsection{The kernel of the module action}

For each $s\in\{p,q,r\}$, let
$$
\rho_s\colon
G_s\to
\operatorname{Aut}_{\ZZ}\bigl(\pi_2(X_s)\bigr)
$$
denote the action homomorphism.

The subgroup
$
\ker\rho_s
=
\{g\in G_s\mid
g\cdot u=u
\text{ for every }
u\in\pi_2(X_s)\}
$
is intrinsic to the $G_s$-module $\pi_2(X_s)$.

Indeed, suppose that
$
\alpha\colon G_s\to G_t
$
is a group isomorphism and
$
\beta\colon\pi_2(X_s)\to\pi_2(X_t)
$
is an additive isomorphism satisfying
$
\beta(g\cdot u)
=
\alpha(g)\cdot\beta(u)
$
for every $g\in G_s$ and $u\in\pi_2(X_s)$. Then necessarily
$
\alpha(\ker\rho_s)=\ker\rho_t.
$

\begin{lemma}
	\label{lem:TT-action-kernel}
	For every $s\in\{p,q,r\}$,
	$
	\ker\rho_s
	=
	\langle\mu_s^s\rangle.
	$
\end{lemma}

\begin{proof}
	Suppose that
	$
	(a,n)\in G_s
	$
	acts trivially on $P$. Applying
	Lemma~\ref{lem:TT-full-action} to $\theta_1$ gives
	$
	\theta_a
	=
	(a,n)\cdot\theta_1
	=
	\theta_1,
	$
	because $\varphi_s^n(1)=1$. By
	\eqref{eq:theta-distinct}, this implies
	$
	a=1.
	$
	
	Thus, every element of $\ker\rho_s$ has the form
	$
	(1,n)=\mu_s^n.
	$
	For every $h\in A$,
	$
	\mu_s^n\cdot\theta_h
	=
	\theta_{\varphi_s^n(h)}.
	$
	By \eqref{eq:theta-distinct}, this action is trivial if and only if
	$
	\varphi_s^n(h)=h
	$
	for every $h\in A$, equivalently,
	$
	\varphi_s^n=\operatorname{id}_A.
	$
	By Theorem~\ref{thm:TT-basic},
	$
	\operatorname{ord}(\varphi_s)=s.
	$
	Therefore
	$
	\varphi_s^n=\operatorname{id}_A
	\;\Longleftrightarrow\;
	s\mid n.
	$
	Hence
	$
	\ker\rho_s=\langle\mu_s^s\rangle.
	$
\end{proof}

\begin{theorem}[The Tanaka--Taniguchi second homotopy modules are distinct]
	\label{thm:TT-pi2-modules-distinct}
	Let
	$
	s,t\in\{p,q,r\},
	\;
	s\neq t.
	$
	There do not exist a group isomorphism
	$
	\alpha\colon G_s\to G_t
	$
	and an additive isomorphism
	$
	\beta\colon
	\pi_2(X_s)
	\to
	\pi_2(X_t)
	$
	such that
	$
	\beta(g\cdot u)
	=
	\alpha(g)\cdot\beta(u)
	$
	for every
	$
	g\in G_s,
	\;
	u\in\pi_2(X_s).
	$
	Consequently, although the knot groups $G_p,G_q,G_r$ are mutually
	isomorphic and the groups
	$
	\pi_2(X_p),\;
	\pi_2(X_q),\;
	\pi_2(X_r)
	$
	are mutually isomorphic as abstract abelian groups, their second homotopy
	modules are pairwise non-isomorphic.
\end{theorem}

\begin{proof}
	Suppose, to the contrary, that compatible isomorphisms
	$$
	\alpha\colon G_s\to G_t,
	\;
	\beta\colon\pi_2(X_s)\to\pi_2(X_t)
	$$
	exist.
	
	Compatibility with the module actions implies
	$
	\alpha(\ker\rho_s)=\ker\rho_t.
	$
	By Lemma~\ref{lem:TT-action-kernel},
	$
	\alpha\bigl(\langle\mu_s^s\rangle\bigr)
	=
	\langle\mu_t^t\rangle.
	$
	
	Let
	$
	\varepsilon_s\colon G_s\to\ZZ,
	\;
	\varepsilon_t\colon G_t\to\ZZ
	$
	be the abelianization homomorphisms normalized by
	$
	\varepsilon_s(\mu_s)=1,
	\;
	\varepsilon_t(\mu_t)=1.
	$
	The abelianization of every $2$-knot group is infinite cyclic,
	generated by a meridian. Hence, the isomorphism induced by $\alpha$
	on abelianizations is multiplication by $1$ or by $-1$. Therefore,
	$
	\varepsilon_t\circ\alpha
	=
	\pm\varepsilon_s.
	$
	
	Now
	$
	\varepsilon_s
	\bigl(
	\langle\mu_s^s\rangle
	\bigr)
	=
	s\ZZ,
	\;
	\varepsilon_t
	\bigl(
	\langle\mu_t^t\rangle
	\bigr)
	=
	t\ZZ.
	$
	Using
	$
	\alpha\bigl(\langle\mu_s^s\rangle\bigr)
	=
	\langle\mu_t^t\rangle
	$
	we obtain
	$$
		t\ZZ
		=
		\varepsilon_t
		\bigl(
		\langle\mu_t^t\rangle
		\bigr)	
		=
		\varepsilon_t
		\left(
		\alpha
		\bigl(
		\langle\mu_s^s\rangle
		\bigr)
		\right)
		=
		\pm
		\varepsilon_s
		\bigl(
		\langle\mu_s^s\rangle
		\bigr)
		=
		s\ZZ.
	$$
	Thus,
	$
	s\ZZ=t\ZZ.
	$
	Since $s,t>0$, this implies
	$
	s=t,
	$
	contrary to the assumption.
	
	Therefore, no compatible isomorphism of the second homotopy modules exists
	when $s\neq t$.
\end{proof}

\begin{corollary}
	\label{cor:TT-pi2-conclusion}
	The Tanaka--Taniguchi examples satisfy
	$
	G_p\cong G_q\cong G_r
	$
	and
	$
	\pi_2(X_p)\cong
	\pi_2(X_q)\cong
	\pi_2(X_r)
	$
	as abstract abelian groups, but
	$
	\pi_2(X_p),\;
	\pi_2(X_q),\;
	\pi_2(X_r)
	$
	are pairwise inequivalent as second homotopy modules under
	isomorphisms of their knot groups.
	
	In particular, the Tanaka--Taniguchi family does not give examples in
	which the fundamental quandle distinguishes $2$-knots having the same
	compatible pair
	$
	(\pi_1,\pi_2).
	$
	Consequently, the compatible pairs
	$
	\bigl(G_s,\pi_2(X_s)\bigr),
	\; s\in\{p,q,r\},
	$
	are pairwise non-isomorphic: there is no group isomorphism between two
	of the knot groups with respect to which the corresponding second
	homotopy modules are semilinearly isomorphic.
\end{corollary}

\begin{proof}
	Combine Theorem~\ref{thm:TT-basic},
	Proposition~\ref{prop:TT-common-abstract-pi2}, and
	Theorem~\ref{thm:TT-pi2-modules-distinct}.
\end{proof}

\begin{remark}
	\label{rem:TT-versus-PS}
	This places the Tanaka--Taniguchi examples in a different position from
	the Plotnick--Suciu examples considered below. In the
	Tanaka--Taniguchi family,
	$
	\pi_1
	\;\text{agrees, whereas}\;
	\pi_2\text{ as a }\ZZ[\pi_1]\text{-module}
	\;\text{already differs}.
	$
	For the Plotnick--Suciu pair, by contrast, both $\pi_1$ and the
	$\ZZ[\pi_1]$-module $\pi_2$ agree, and we will show that the
	fundamental quandles agree as well. Thus, the Plotnick--Suciu
	construction is the relevant setting for asking whether the additional
	quandle data determine the first Postnikov invariant once the compatible
	pair $(\pi_1,\pi_2)$ has already been fixed.
\end{remark}

\section{The Plotnick--Suciu examples}
\label{sec:plotnick-suciu}

Plotnick and Suciu begin with two fibered $2$-knots $K_1,K_2$ whose
punctured fibers are
$
\Sigma_1^\circ
=
\Sigma_1\setminus\operatorname{int}\BB^3,
\;
\Sigma_2^\circ
=
\Sigma_2\setminus\operatorname{int}\BB^3,
$
where $\Sigma_1,\Sigma_2$ are closed, oriented, aspherical
$3$-manifolds. Put
$
A_i=\pi_1(\Sigma_i),
\;
H_i=A_i\rtimes_{\sigma_i}\langle x\rangle,
\;
H=H_1*_{\langle x\rangle}H_2.
$

To avoid ambiguity, let $K_1^\dagger$ denote the knot denoted
$-K_1$ in \cite[\S~1]{PS85}: namely, the mirror image obtained by
applying an orientation-reversing diffeomorphism of the ambient
$\s^4$. This is not merely a reversal of the orientation assigned to
the embedded $2$-sphere. Its closed fiber is $-\Sigma_1$, where
$-\Sigma_1$ denotes $\Sigma_1$ with the opposite orientation, while
under the natural identification the monodromy is still $\sigma_1$.

Let
$
\sigma=\sigma_1\#\sigma_2
$
denote the connected-sum monodromy, and put
$
Y=(\Sigma_1\#\Sigma_2)\times_\sigma \s^1,
\;
Y'=((-\Sigma_1)\#\Sigma_2)\times_\sigma \s^1.
$
Thus
$
\pi_1(Y)\cong H\cong\pi_1(Y').
$
Let
$
W=\s^1_t\times\BB^3.
$
The two Plotnick--Suciu exteriors are obtained from
$
W\mathbin{\#}Y
\;\text{and}\;
W\mathbin{\#}Y',
$
respectively, by surgery on an interior loop representing
$
r=txt^{-1}x^{-2}.
$
We denote the resulting exteriors by $\widehat X$ and $\widehat X'$,
respectively. Here the framing of the surgery is chosen so that the result is
the exterior of a $2$-knot in the standard $\s^4$: filling $\widehat X$ with
$\DD^2\times\s^2$ along $\s^1\times\partial\BB^3$ kills $t$, and the resulting
closed manifold is $\s^4\mathbin{\#}Y$ surgered along $x$, that is, $\s^4$
\cite[\S~1, pp.~55--56]{PS85}. Equivalently, $\widehat X$ comes from the input
$K_1\#K_2$, while $\widehat X'$ comes from $K_1^\dagger\#K_2$
\cite[\S~1]{PS85}.

\begin{theorem}[Plotnick--Suciu]
	\label{thm:plotnick-suciu}
	Let
	$
	G=\langle t,x\mid txt^{-1}=x^2\rangle,
	\;
	\Pi=G*_{\langle x\rangle}H.
	$
Then
$
\pi_1(\widehat X)\cong\Pi\cong\pi_1(\widehat X').
$
Moreover, there exist an isomorphism
$
\alpha_0\colon
\pi_1(\widehat X)\xrightarrow{\cong}\pi_1(\widehat X')
$
and an $\alpha_0$-semilinear isomorphism
$
\beta_0\colon
\pi_2(\widehat X)\xrightarrow{\cong}\pi_2(\widehat X'),
$
that is,
$
\beta_0(g\cdot u)
=
\alpha_0(g)\cdot\beta_0(u)
$
for all $g\in\pi_1(\widehat X)$ and
$u\in\pi_2(\widehat X)$
\cite[Proposition~3.2 and Theorem~1.1]{PS85}.

If $\Sigma_1$ and $\Sigma_2$ admit no orientation-reversing homotopy
equivalences, then there is no compatible pair
$
(\alpha,\beta),
$
where
$
\alpha\colon
\pi_1(\widehat X)\xrightarrow{\cong}\pi_1(\widehat X')
$
is a group isomorphism and
$
\beta\colon
\pi_2(\widehat X)\xrightarrow{\cong}\pi_2(\widehat X')
$
is an $\alpha$-semilinear module isomorphism, which carries
$k_{\widehat X}$ to $k_{\widehat X'}$
\cite[\S\S6--7]{PS85}. In particular,
$
\widehat X\not\simeq\widehat X'.
$
\end{theorem}

\begin{remark}
	\label{rem:postnikov-orbits}
Plotnick and Suciu compute the first Postnikov invariants of the two
exteriors and then analyze the action of the compatible automorphisms
of the fundamental group and second homotopy module
\cite[\S\S5--7]{PS85}. The point needed here is not a particular
coordinate representative of either cohomology class, but their
conclusion that no compatible group and module automorphisms carry
$k_{\widehat X}$ to $k_{\widehat X'}$.
	
	Equivalently, the argument in \cite[\S\S~5--7]{PS85} rules out
	a pair consisting of an automorphism of the common fundamental group
	and a compatible semilinear automorphism of the common second homotopy
	module which carries $k_{\widehat X}$ to $k_{\widehat X'}$.
	Thus, the two Postnikov classes lie in distinct orbits under compatible
	group and module automorphisms.
	
The conclusion of \cite[Theorem~1.2]{PS85} is stated there as the
non-existence of a map $\widehat X\to\widehat X'$ inducing an
isomorphism on $\pi_1$. By the realization theorem of MacLane and
Whitehead, quoted in \cite[p.~58]{PS85} and stated explicitly in
\cite[p.~3]{Suc84}, this is equivalent to the statement that no
compatible pair $(\alpha,\beta)$ carries $k_{\widehat X}$ to
$k_{\widehat X'}$, which is the form used here; indeed,
equations~(2)--(4) of \cite[\S7]{PS85} are exactly the condition that
$\alpha$ and $\beta$ preserve $k$-invariants.
\end{remark}

\subsection{Brieskorn specialization and the common meridian}

Let $n_1,n_2>5$ be distinct integers, each coprime to $6$, and set
$
\Sigma_1=\Sigma(2,3,n_1),
\;
\Sigma_2=\Sigma(2,3,n_2).
$
For $i=1,2$, take
$
K_i=\tau^2(T_{3,n_i}),
$
the $2$-twist spin of the $(3,n_i)$-torus knot. By Zeeman's
twist-spinning theorem \cite{Zee65}, $K_i$ is fibered, with punctured
fiber
$
\Sigma_i^\circ
=
\Sigma(2,3,n_i)\setminus\operatorname{int}\BB^3
$
and monodromy induced by the deck involution of the double branched
cover
$
\Sigma(2,3,n_i)\to \s^3
$
branched over $T_{3,n_i}$. 

The deck involution has a one-dimensional fixed-point set. A sufficiently
small invariant $3$-ball centered at a fixed point is equivariantly
diffeomorphic to the standard rotation ball, so its restriction to the
boundary $\s^2$ is conjugate to the order-two rotation through angle
$\pi$. Thus the invariant balls for the two inputs may be chosen so
that their boundary actions agree, and the connected-sum monodromy
$
\sigma_1\#\sigma_2
$
is well defined.

The manifolds
$
\Sigma_i=\Sigma(2,3,n_i),
\; i=1,2,
$
are integral homology $3$-spheres and, since $n_i>5$, are aspherical
Seifert fibered $3$-manifolds. By the proof of \cite[Theorem~8.2]{NR78}, an orientation-reversing homotopy
equivalence of a closed oriented aspherical Seifert $3$-manifold
forces its Seifert Euler number to vanish. For the Brieskorn homology
sphere $\Sigma(2,3,n_i)$ one has
$
\bigl|e(\Sigma(2,3,n_i))\bigr|
=
\frac{1}{6n_i},
$
with the sign depending on the orientation convention. In particular,
its Seifert Euler number is nonzero. Hence neither $\Sigma_i$ admits
an orientation-reversing homotopy equivalence. This verifies the
hypothesis of Theorem~\ref{thm:plotnick-suciu}; compare also
\cite[pp.~56--57]{PS85}.

Let $\widehat K$ and $\widehat K'$ denote the oriented $2$-knots with
exteriors $\widehat X$ and $\widehat X'$, respectively. The first
construction uses the closed fibers $\Sigma_1,\Sigma_2$, while the
primed construction uses $-\Sigma_1,\Sigma_2$.

The next lemma makes basepoint-explicit the geometric input needed for
Theorem~\ref{thm:brieskorn-same-quandle} below: the distinguished boundary
meridian coming from the common $W$-piece can be represented by the same
literal element $t\in\Pi$ in both exteriors.

\begin{lemma}[The connected-sum piece is shared by the two constructions]
	\label{lem:shared-piece}
	Let
	$
	W=\s^1_t\times\BB^3,
	$
	fix $p'\in\s^2=\partial\BB^3$, and set
	$
	\partial W=\s^1_t\times\s^2,
	\;
	\mu_t=\s^1_t\times\{p'\},
	\;
	p=(1,p')\in\mu_t\subset\partial W,
	$
	so that $\pi_1(W,p)=\langle t\rangle$ with $t=[\mu_t]$. As recalled in
	\S~\ref{sec:plotnick-suciu}, $\widehat X$ is obtained from
	$W\mathbin{\#}Y$ by surgery on an interior curve representing
	$
	r=txt^{-1}x^{-2},$ with the framing fixed above,
	and $\widehat X'$ is obtained analogously from
	$W\mathbin{\#}Y'$ by surgery on an interior curve representing the same
	word $r$. Then $p$ is a common basepoint for $\widehat X$ and
	$\widehat X'$, and there are based identifications
	$
	\pi_1(\widehat X,p)\cong\Pi
	\;\text{and}\;
	\pi_1(\widehat X',p)\cong\Pi,
	$
	under which the based loop $\mu_t$
	is sent to the same literal element $t\in\Pi$ in both groups.
\end{lemma}

\begin{proof}
	The connected-sum operation and the surgery are performed in the
	interior, so neither changes $\partial W=\s^1_t\times\s^2$. Hence
	$(\partial W,p)$ and the based loop $\mu_t$ include unchanged into both
	$\widehat X$ and $\widehat X'$.
	
	For $\widehat X$, van Kampen's theorem, after transporting the
	fundamental group of the $Y$-summand to the basepoint $p$ through the
	connected-sum neck, gives
	$
	\pi_1(W\mathbin{\#}Y,p)\cong \langle t\rangle * H.
	$
	Surgery on a curve representing $r=txt^{-1}x^{-2}$ then gives
	$
	\pi_1(\widehat X,p)
	\cong
	(\langle t\rangle * H)/
	\langle\!\langle txt^{-1}x^{-2}\rangle\!\rangle.
	$
	The presentation on the right is precisely
	$
	G*_{\langle x\rangle}H=\Pi,
	\;
	G=\langle t,x\mid txt^{-1}=x^2\rangle,
	$
	and the resulting isomorphism may be chosen to be the identity on the
	named generators. In particular, it sends the element
	$t=[\mu_t]\in\pi_1(W,p)$ to the element $t\in G\leq\Pi$.
	
	The same van Kampen computation applies to $\widehat X'$ using the
	common identification of the corresponding $Y'$-factor with $H$ from
	the Plotnick--Suciu calculation. Again the resulting based
	isomorphism with $\Pi$ may be chosen to be the identity on the named
	generators, and therefore again sends $[\mu_t]$ to $t\in\Pi$.
\end{proof}

\begin{remark}[Independence of the chosen identifications]
	There are two distinct uses of group identifications here.
	The based identifications in Lemma~\ref{lem:shared-piece} are used
	only to identify the distinguished positive meridians and hence the
	fundamental quandles. We do not assert that, under those particular
	identifications, the two second homotopy modules become isomorphic
	by the identity automorphism of $\Pi$.
	
	The Plotnick--Suciu module and Postnikov statements are instead
	formulated intrinsically: there exists a group isomorphism together
	with a semilinear isomorphism of the second homotopy modules, whereas
	no such compatible pair carries one first Postnikov invariant to the
	other. This assertion is unchanged by replacing either coordinate
	identification with $\Pi$ by an automorphism of $\Pi$.
\end{remark}

\begin{theorem}[The Brieskorn Plotnick--Suciu knots have the same quandle]
	\label{thm:brieskorn-same-quandle}
	For the Brieskorn specialization above, fix an orientation of the
	distinguished boundary loop
	$
	\mu_t=\s^1_t\times\{p'\},
	$
	and orient $\widehat K$ and $\widehat K'$ so that $\mu_t$ is the
	positive meridian of each knot. Under the based identifications
	$
	\pi_1(\widehat X,p)\cong\Pi
	\;\text{and}\;
	\pi_1(\widehat X',p)\cong\Pi
	$
	fixed in Lemma~\ref{lem:shared-piece}, the two positive meridians are
	represented by the same element $t\in\Pi$:
	$
	m_{\widehat K}=t=m_{\widehat K'}.
	$
	Consequently,
	$
	Q(\widehat K)
	\cong
	\Cos\bigl(\Pi,\langle t\rangle,t\bigr)
	\cong
	Q(\widehat K').
	$
	Moreover, there is a quandle isomorphism
$f\colon Q(\widehat K)\to Q(\widehat K')$ such that, under the based
identifications with $\Pi$, the induced automorphism
$
\eta_{\widehat K'}\circ\As(f)\circ\eta_{\widehat K}^{-1}
\colon \Pi\to\Pi
$
is the identity.
\end{theorem}

\begin{proof}
	By the construction of \cite[\S~1]{PS85},
	$\mu_t$ is a meridian of each of the two final $2$-knots. Since reversing the orientation of a $2$-knot reverses the
	orientation of its positive meridian, each of $\widehat K$ and
	$\widehat K'$ has a unique choice of orientation for which the fixed
	orientation of $\mu_t$ is positive.
	
	With these orientations fixed, Lemma~\ref{lem:shared-piece} gives
	based identifications with $\Pi$ under which
	$
	[\mu_t]\mapsto t
	$
	in both groups. Hence
	$
	m_{\widehat K}=t=m_{\widehat K'},
	$
	and both oriented peripheral triples are identified with
	$
	\bigl(\Pi,\langle t\rangle,t\bigr).
	$
	Proposition~\ref{prop:quandle-peripheral} therefore gives
	$
	Q(\widehat K)
	\cong
	\Cos\bigl(\Pi,\langle t\rangle,t\bigr)
	\cong
	Q(\widehat K').
	$
		For the final assertion, let
	$
	\kappa_{\widehat K}\colon
	Q(\widehat K)\to
	\Cos(\Pi,\langle t\rangle,t),
	\;
	\kappa_{\widehat K'}\colon
	Q(\widehat K')\to
	\Cos(\Pi,\langle t\rangle,t)
	$
	be the coset identifications supplied by
	Proposition~\ref{prop:quandle-peripheral}, and put
	$
	f=
	\kappa_{\widehat K'}^{-1}\circ
	\kappa_{\widehat K}.
	$
	If
	$
	\kappa_{\widehat K}(q)=\langle t\rangle g,
	$
	then, by the construction of the peripheral coset description,
	$
	\eta_{\widehat K}(e_q)=g^{-1}tg.
	$
	Since $f(q)$ corresponds to the same coset
	$\langle t\rangle g$ for $\widehat K'$, we likewise have
	$
	\eta_{\widehat K'}(e_{f(q)})=g^{-1}tg.
	$
	Therefore
	$
	\eta_{\widehat K'}\circ\As(f)
	=
	\eta_{\widehat K}
	$
	on every canonical generator of
	$\As(Q(\widehat K))$, and hence on the whole associated group.
	Thus
	$
	\eta_{\widehat K'}\circ\As(f)\circ
	\eta_{\widehat K}^{-1}
	=
	\operatorname{id}_{\Pi}.
	$
\end{proof}

\begin{theorem}
	\label{thm:main-PS}
	For every pair of distinct integers $n_1,n_2>5$, each coprime to $6$,
	the Brieskorn Plotnick--Suciu construction above yields oriented
	$2$-knots $\widehat K$ and $\widehat K'$ with the following
	properties:
	\begin{enumerate}
		\item
		their fundamental groups are isomorphic;
		
		\item
		there exist an isomorphism
		$
		\alpha_0\colon
		\pi_1(\widehat X)\xrightarrow{\cong}\pi_1(\widehat X')
		$
		and an $\alpha_0$-semilinear isomorphism
		$
		\beta_0\colon
		\pi_2(\widehat X)\xrightarrow{\cong}\pi_2(\widehat X');
		$
		
		\item
		their fundamental quandles are isomorphic:
		$
		Q(\widehat K)\cong Q(\widehat K');
		$
		
		\item
		no compatible pair $(\alpha,\beta)$, consisting of a group
		isomorphism
		$
		\alpha\colon
		\pi_1(\widehat X)\xrightarrow{\cong}\pi_1(\widehat X')
		$
		and an $\alpha$-semilinear isomorphism
		$
		\beta\colon
		\pi_2(\widehat X)\xrightarrow{\cong}\pi_2(\widehat X'),
		$
		carries $k_{\widehat X}$ to $k_{\widehat X'}$.
	\end{enumerate}
	
	In particular,
	$
	\widehat X\not\simeq\widehat X'.
	$
	
Thus, the isomorphism class of the compatible pair
$(\pi_1,\pi_2)$, together with the isomorphism class of the fundamental
quandle, does not determine the homotopy type of an oriented
$2$-knot exterior; in particular, these invariant classes do not
determine its first Postnikov invariant.
	
	Moreover, the isomorphism
	$
	f\colon Q(\widehat K)\xrightarrow{\cong}Q(\widehat K')
	$
	in {\rm(3)} may be chosen so that, under the based identifications
	of Lemma~\ref{lem:shared-piece},
	$
	\eta_{\widehat K'}
	\circ
	\As(f)
	\circ
	\eta_{\widehat K}^{-1}
	=
	\operatorname{id}_{\Pi}.
	$
\end{theorem}

\begin{proof}
	Assertions {\rm(1)}, {\rm(2)}, and {\rm(4)} are the conclusions of
	Theorem~\ref{thm:plotnick-suciu} for the Brieskorn specialization
	above. Assertion {\rm(3)} and the final compatibility statement for
	the quandle isomorphism are
	Theorem~\ref{thm:brieskorn-same-quandle}.
	
	If there were a homotopy equivalence
	$
	h\colon\widehat X\to\widehat X',
	$
	then the induced maps
	$
	h_*\colon\pi_1(\widehat X)\xrightarrow{\cong}\pi_1(\widehat X')
	$
	and
	$
	h_*\colon\pi_2(\widehat X)\xrightarrow{\cong}\pi_2(\widehat X')
	$
	would form a compatible semilinear pair and, by naturality of the
	first Postnikov invariant, would carry
	$k_{\widehat X}$ to $k_{\widehat X'}$. This contradicts {\rm(4)}.
	Hence
	$
	\widehat X\not\simeq\widehat X'.
	$
\end{proof}


\subsection{Suciu's lens-space family and arbitrarily many examples}
\label{subsec:suciu-many}

We next use the alternative construction from Suciu's thesis
\cite[Chapter~III]{Suc84}. Throughout that construction, $p$ is odd
and
$
0<q<p,
\;
\gcd(p,q)=1.
$
For such $p$ and $q$, let $K_{p,q}$ denote the $2$-knot constructed
in \cite[Chapter~III, \S3]{Suc84}, and let $X_{p,q}$ be its exterior.

For fixed $p$, Suciu identifies the fundamental groups of the
$X_{p,q}$ with the same group
$
\pi_p
=
\left\langle
t,a,x
\;\middle|\;
a^p=1,\;
xax^{-1}=a^{-1},\;
txt^{-1}=x^2
\right\rangle
=
G*_{\langle x\rangle}H_p,
$
where
$
G=\langle t,x\mid txt^{-1}=x^2\rangle
$
and
$
H_p
=
\langle a,x\mid a^p=1,\;xax^{-1}=a^{-1}\rangle.
$
The parameter $q$ enters through the internal lens-space piece
$Y_{p,q}$ and not through this presentation
\cite[Chapter~III, \S3]{Suc84}. Moreover, Suciu computes
$\pi_2(X_{p,q})$ as a $\ZZ[\pi_p]$-module by a presentation
independent of $q$; see
\cite[Chapter~III, Proposition~4.1]{Suc84}.

\begin{corollary}[Arbitrarily many]
	\label{cor:arbitrarily-many}
	For every integer $N\geq2$, there exist an odd integer $p$ and
	integers
	$
	0<q_1,\dots,q_N<p,
	\;
	\gcd(p,q_i)=1,
	$
	such that, putting
	$
	\pi=\pi_p,
	$
	the oriented $2$-knots
	$
	L_1,\dots,L_N,
	\;
	L_i=K_{p,q_i},
	$
	have the following properties:
	\begin{enumerate}
		\item
		$
		\pi_1(X_{L_i})\cong\pi
		$
		for every $i$;
		
		\item
for every $i,j$ there exist a group isomorphism
$
\alpha_{ij}\colon
\pi_1(X_{L_i})\xrightarrow{\cong}\pi_1(X_{L_j})
$
and an $\alpha_{ij}$-semilinear isomorphism
$
\beta_{ij}\colon
\pi_2(X_{L_i})\xrightarrow{\cong}\pi_2(X_{L_j});
$
		
		\item their fundamental quandles are mutually isomorphic; more
		precisely,
		$
		Q(L_i)
		\cong
		\Cos(\pi,\langle t\rangle,t)
		\;
		\text{for every }i;
		$
		
		\item
for $i\neq j$, no compatible pair $(\alpha,\beta)$ consisting
of a group isomorphism
$
\alpha\colon
\pi_1(X_{L_i})\xrightarrow{\cong}\pi_1(X_{L_j})
$
and an $\alpha$-semilinear isomorphism
$
\beta\colon
\pi_2(X_{L_i})\xrightarrow{\cong}\pi_2(X_{L_j})
$
carries $k_{X_{L_i}}$ to $k_{X_{L_j}}$.
	\end{enumerate}
	
	In particular, the exteriors $X_{L_1},\dots,X_{L_N}$ are pairwise
	non-homotopy-equivalent.
\end{corollary}

\begin{proof}
Suciu proves that, for every $N\geq2$, one can choose $p$ and
$q_1,\dots,q_N$ as above so that the exteriors
$
X_{p,q_1},\dots,X_{p,q_N}
$
have fundamental groups isomorphic to $\pi_p$; that for all $i,j$
there are a group isomorphism
$\pi_1(X_{p,q_i})\to\pi_1(X_{p,q_j})$ together with a compatible
semilinear isomorphism of the second homotopy modules; and that no
map between distinct members realizes compatible isomorphisms on
$\pi_1$ and $\pi_2$;
see \cite[Theorem~1.1 and Chapter~III, Theorem~3.2 and \S7]{Suc84}.
	
	By the MacLane--Whitehead realization theorem, stated in
	\cite[p.~3]{Suc84}, the latter assertion is equivalent to saying
	that the corresponding first Postnikov invariants lie in distinct
	orbits under compatible group and module isomorphisms.
	
	It remains only to compare the fundamental quandles.
	Each $X_{p,q_i}$ is obtained by surgery on the same word
	$
	r=txt^{-1}x^{-2}
	$
	in
	$
	W\mathbin{\#}Y_{p,q_i},
	\;
	W=\s^1_t\times\BB^3;
	$
	see \cite[Chapter~III, \S3]{Suc84}.
	Hence Lemma~\ref{lem:shared-piece} applies verbatim, with
	$Y_{p,q_i}$ in place of $Y$ and $\pi_p$ in place of $\Pi$:
	under the resulting based identification
	$
	\pi_1(X_{p,q_i})\cong\pi_p,
	$
	the distinguished boundary meridian
	$
	\mu_t=\s^1_t\times\{p'\}
	$
	is represented by the same literal element $t\in\pi_p$ for every
	$i$.
	
	Orient each $L_i=K_{p,q_i}$ so that the fixed orientation of
	$\mu_t$ is positive. The orientation argument in
	Theorem~\ref{thm:brieskorn-same-quandle} then gives the common
	oriented peripheral triple
	$
	\bigl(\pi_p,\langle t\rangle,t\bigr).
	$
	Therefore Proposition~\ref{prop:quandle-peripheral} gives
	$
	Q(L_i)
	\cong
	\Cos(\pi_p,\langle t\rangle,t)
	$
	for every $i$.
	The final assertion follows from {\rm(4)} exactly as in the proof of
	Theorem~\ref{thm:main-PS}.
\end{proof}

\begin{remark}
	Plotnick and Suciu also describe in \cite[\S8]{PS85} a more
	structured multi-fiber construction producing $2^{n-1}$ examples
	from $n$ fibers. The proof of \cite[Theorem~8.1]{PS85} is given
	there as a sketch. Corollary~\ref{cor:arbitrarily-many} does not
	depend on that sketch: the existence of arbitrarily large families
	comes instead from the complete lens-space argument in
	\cite[Chapter~III, \S\S3--7]{Suc84}.
\end{remark}


\subsection{The Alexander module of the examples}

By Proposition~\ref{prop:quandle-determines-Alexander} and
Theorem~\ref{thm:brieskorn-same-quandle}, the Alexander modules of
$\widehat K$ and $\widehat K'$ are already known to be isomorphic.
We now identify their common isomorphism type explicitly.

\begin{lemma}[The Alexander module of the Plotnick--Suciu knots]
	\label{lem:PS-alexander-module}
	Let
	$
	\Pi=G*_{\langle x\rangle}H
	$
	be as above, where
	$
	G=\operatorname{BS}(1,2)
	=
	\langle t,x\mid txt^{-1}=x^2\rangle,
	$
	$
	H=H_1*_{\langle x\rangle}H_2,
	\;
	H_i=A_i\rtimes_{\sigma_i}\langle x\rangle,
	$
	and $A_i=\pi_1(\Sigma_i)$, with $\Sigma_i$ integral homology
	$3$-spheres. Let
	$
	\varepsilon\colon\Pi\to\ZZ
	$
	be the abelianisation normalized by $\varepsilon(t)=1$. Then
	$
	\Alex(\widehat K)
	\cong
	\Alex(\widehat K')
	\cong
	\Lambda/(t-2),
	$
	and consequently
	$
	\Delta_{\widehat K}(t)
	\doteq
	\Delta_{\widehat K'}(t)
	\doteq
	t-2.
	$
\end{lemma}

\begin{proof}
	Since
	$
	A_i^{\mathrm{ab}}
	=
	H_1(\Sigma_i;\ZZ)
	=
	0,
	$
	each $A_i$ is perfect. Hence
	$
	H_i^{\mathrm{ab}}
	\cong
	\langle x\rangle
	\cong
	\ZZ,
	\;
	H^{\mathrm{ab}}
	\cong
	\langle x\rangle
	\cong
	\ZZ.
	$
	In $G$, the relation $txt^{-1}=x^2$ gives
	$
	[x]=2[x]
	$
	after abelianisation. Thus, $[x]=0$ and
	$
	G^{\mathrm{ab}}
	\cong
	\langle[t]\rangle
	\cong
	\ZZ.
	$
	It follows that
	$
	\varepsilon(x)=0,
	\;
	\varepsilon(H)=0.
	$
	
	Put $C=\langle x\rangle$. Choose connected CW-complexes
	$Y_G,Y_H,Y_C$ with fundamental groups $G,H,C$, respectively,
	and maps
	$
	Y_C\to Y_G,
	\;
	Y_C\to Y_H
	$
	inducing the two inclusions of $C$. Replacing these maps by mapping
	cylinders if necessary, form a graph-of-spaces model
	$
	Y=Y_G\cup_{Y_C}Y_H
	$
	with
	$
	\pi_1(Y)\cong G*_{C}H=\Pi.
	$
	Let
	$
	\widetilde Y\to Y
	$
	be the cover associated with $\ker\varepsilon$. Since
	$
	\pi_1(\widetilde Y)=\ker\varepsilon,
	$
	we have
	$
	H_1(\widetilde Y;\ZZ)
	\cong
	(\ker\varepsilon)^{\mathrm{ab}}
	\cong
	\Alex(\widehat K)
	$
as $\Lambda$-modules; by the same covering-space argument as in the
proof of Proposition~\ref{prop:quandle-determines-Alexander}, the deck
action corresponds to conjugation by an element of
$\varepsilon$-level one.
	
	Since $\varepsilon$ vanishes on $C$ and on $H$, the inverse images
	of $Y_C$ and $Y_H$ are disjoint unions of copies indexed by $\ZZ$.
	Therefore
	$
	H_1(\widetilde Y_C;\ZZ)
	\cong
	\Lambda\otimes_{\ZZ}H_1(Y_C;\ZZ)
	\cong
	\Lambda,
	$
	$
	H_1(\widetilde Y_H;\ZZ)
	\cong
	\Lambda\otimes_{\ZZ}H_1(Y_H;\ZZ)
	\cong
	\Lambda,
	$
	and
	$
	H_0(\widetilde Y_C;\ZZ)
	\cong
	H_0(\widetilde Y_H;\ZZ)
	\cong
	\Lambda.
	$
	
	On the $G$-piece, put
	$
	x_n=t^{-n}xt^n,
	\; n\geq0.
	$
	The defining relation gives
	$
	x_n^2=x_{n-1}
	\; (n\geq1).
	$
		Hence
	$
	\langle x\rangle
	\subset
	\langle x_1\rangle
	\subset
	\langle x_2\rangle
	\subset\cdots.
	$
	
	To identify this union, consider the affine action
	$
	G\to \operatorname{Homeo}^+(\mathbb R)
	$
	defined by
	$
	x(u)=u+1,
	\;
	t(u)=2u.
	$
	The relation $txt^{-1}=x^2$ is satisfied. Moreover,
	$
	x_n=t^{-n}xt^n
	$
	acts as
	$
	x_n(u)=u+2^{-n}.
	$
	Thus, every $x_n$ has infinite order. Since every element of
	$
	D:=\bigcup_{n\geq0}\langle x_n\rangle
	$
	belongs to one of the cyclic groups $\langle x_n\rangle$, the
	restriction of this affine action to $D$ is injective, and its image
	is precisely the group of translations by elements of
	$
	\bigcup_{n\geq0}2^{-n}\ZZ=\ZZ[1/2].
	$
	Therefore
	$
	D\cong\ZZ[1/2].
	$
	The union $D$ is normal in $G$: it is preserved by conjugation by $t$
	by construction, and, being abelian, by conjugation by $x$. Since it
	contains $x$ and each $x_n$ is a conjugate of $x$, it coincides with
the normal closure of $x$. Since $G/\langle\!\langle x\rangle\!\rangle
\cong\langle t\rangle\cong\ZZ$, we get
	$
	\ker(\varepsilon|_G)=D\cong\ZZ[1/2].
	$
	Conjugation by $t$ acts on $D$ by multiplication by $2$, so
	$
	H_1(\widetilde Y_G;\ZZ)
	\cong
	\Lambda/(t-2).
	$
	Moreover $\widetilde Y_G$ is connected, and hence
	$
	H_0(\widetilde Y_G;\ZZ)
	\cong
	\Lambda/(t-1).
	$
	
	The Mayer--Vietoris sequence contains
	$
	\Lambda
	\xrightarrow{\;\phi\;}
	\Lambda/(t-2)\oplus\Lambda
	\to
	H_1(\widetilde Y;\ZZ)
	\to
	\Lambda
	\xrightarrow{\;\psi\;}
	\Lambda/(t-1)\oplus\Lambda.
	$
	With the standard Mayer--Vietoris sign convention, the generators
	may be chosen so that
	$
	\phi(1)=(\overline1,-1),
	\;
	\psi(1)=(\overline1,-1).
	$
	The second component of $\psi$ is $-\operatorname{id}_{\Lambda}$,
	so $\psi$ is injective. Consequently,
	$
	H_1(\widetilde Y;\ZZ)
	\cong
	\operatorname{coker}\phi.
	$
	The homomorphism
	$
	\Lambda/(t-2)\oplus\Lambda
	\to
	\Lambda/(t-2),
	\;
	(a,b)\longmapsto a+\overline b,
	$
	has kernel precisely
	$
	\Lambda(\overline1,-1).
	$
	Therefore
	$
	H_1(\widetilde Y;\ZZ)
	\cong
	\Lambda/(t-2).
	$
	Hence
	$
	\Alex(\widehat K)
	\cong
	\Lambda/(t-2),
	\;
	\Delta_{\widehat K}(t)\doteq t-2.
	$
	
	Finally, Theorem~\ref{thm:brieskorn-same-quandle} and
	Proposition~\ref{prop:quandle-determines-Alexander} give
	$
	\Alex(\widehat K')
	\cong
	\Alex(\widehat K),
	$
	and therefore
	$
	\Alex(\widehat K')
	\cong
	\Lambda/(t-2),
	\;
	\Delta_{\widehat K'}(t)\doteq t-2.
	$
\end{proof}

\begin{corollary}
	\label{cor:suciu-alexander}
	For every member $K_{p,q}$ of Suciu's lens-space family above,
	$
	\Alex(K_{p,q})
	\cong
	\Lambda/(t-2),
	\;
	\Delta_{K_{p,q}}(t)\doteq t-2.
	$
	In particular, all the knots
	$L_1,\dots,L_N$ of
	Corollary~\ref{cor:arbitrarily-many} have the same Alexander module
	and Alexander polynomial.
\end{corollary}

\begin{proof}
	For
	$
	H_p
	=
	\langle a,x\mid a^p=1,\;xax^{-1}=a^{-1}\rangle,
	$
	abelianisation gives
	$
	2[a]=0,
	\;
	p[a]=0.
	$
	Since $p$ is odd, $[a]=0$, and hence
	$
	H_p^{\mathrm{ab}}
	\cong
	\langle[x]\rangle
	\cong
	\ZZ.
	$
	Thus the inclusion
	$
	\langle x\rangle\to H_p
	$
	induces an isomorphism on abelianisations. The Mayer--Vietoris
	calculation in the proof of
	Lemma~\ref{lem:PS-alexander-module} therefore applies verbatim to
	$
	\pi_p=G*_{\langle x\rangle}H_p,
	$
	and gives
	$
	\Alex(K_{p,q})\cong\Lambda/(t-2).
	$
	The assertion about the Alexander polynomial follows.
\end{proof}

\begin{remark}
	The answer is consistent with a Fox-calculus computation on the
	$G$-piece alone. If the summand $Y$ is replaced by $\s^1\times\s^3$,
	the same surgery produces a ribbon $2$-knot with group
	$\langle t,x\mid r\rangle=G$ \cite[\S3]{PS85}, whose Alexander module
	is $\Lambda/\bigl(\overline{\partial r/\partial x}\bigr)$, the bar
	denoting the map $t\mapsto t$, $x\mapsto1$. For
	$r=txt^{-1}x^{-2}$,
	$
	\frac{\partial r}{\partial x}
	=
	t+txt^{-1}\bigl(-x^{-1}-x^{-2}\bigr),
	\;
	\overline{\frac{\partial r}{\partial x}}
	=
	t-2 .
	$
	That this agrees with $\Alex(\widehat K)$ is also transparent from
	the Mayer--Vietoris calculation above. Since the groups $A_i$ are
	perfect,
	$
	H^{\mathrm{ab}}\cong\langle x\rangle\cong\ZZ,
	$
	and the inclusion
	$
	C=\langle x\rangle\to H
	$
	induces an isomorphism on abelianisations. Consequently, the
	$\Lambda$-summand contributed by the $H$-piece is cancelled by the
	corresponding edge-group summand, leaving precisely the contribution
	$\Lambda/(t-2)$ from the $G$-piece.
\end{remark}

\begin{remark}
	Proposition~\ref{prop:quandle-determines-Alexander} explains
	conceptually why the two Alexander modules in
	Lemma~\ref{lem:PS-alexander-module} must agree:
	the argument of Theorem~\ref{thm:brieskorn-same-quandle} gives
	$
	Q(\widehat K)\cong Q(\widehat K'),
	$
	and the Alexander module is determined by the fundamental quandle.
	Lemma~\ref{lem:PS-alexander-module} adds the explicit identification
	$
	\Alex(\widehat K)
	\cong
	\Alex(\widehat K')
	\cong
	\Lambda/(t-2).
	$
	Thus, the Plotnick--Suciu pair agrees not only in its compatible
	$(\pi_1,\pi_2)$ data and fundamental quandle, but also in its
	Alexander module and Alexander polynomial. Nevertheless, the first
	Postnikov invariants remain inequivalent by
	Theorem~\ref{thm:main-PS}.
\end{remark}

\section*{Acknowledgements}

This work was partially carried out during the 2026 IAS/Park City Mathematics Institute program ``Knotted Surfaces in Four-Manifolds''.
The author thanks PCMI for its stimulating environment.

\end{document}